%% file: mainCameraReady.tex
\documentclass[runningheads]{llncs}
\usepackage[T1]{fontenc}
\usepackage{lmodern}
\usepackage{graphicx}
\usepackage{commands}
\usepackage[normalem]{ulem}
\usepackage[hidelinks]{hyperref}
\usepackage{color}

\begin{document}
\title{Crossing-Preserving Geodesic Tracking on Spherical Images}
%
%

\author{
Nicky J. van den Berg\inst{1}\thanks{Joint main authors.}
\and
Finn M. Sherry\inst{1,2}\repeatthanks
\and
Tos T.J.M. Berendschot\inst{3}
\and 
Remco Duits\inst{1,2}
}
\authorrunning{N.J. van den Berg and F.M. Sherry et al.}
\institute{CASA, Dep. of Mathematics \& Computer Science, Eindhoven University of Technology,  the Netherlands\\
\email{\{n.j.v.d.berg,f.m.sherry,r.duits\}@tue.nl}
\and
EAISI, Eindhoven University of Technology, the Netherlands
\and
Ophthalmology Department Maastricht University, the Netherlands\\
\email{t.berendschot@maastrichtuniversity.nl}
}

\maketitle              
\input{Sections/0_Abstract}
\input{Sections/1_Introduction}
\input{Sections/2_TrackingInSE2}
\input{Sections/3_TrackingInSO3}
\input{Sections/4_PDESolver}
\input{Sections/5_Experiments_NonAnonymous}
\input{Sections/6_Conclusion}

\subsubsection{Acknowledgements}
We gratefully acknowledge the Dutch Foundation of Science NWO for its financial support by Talent Programme VICI 2020 Exact Sciences (Duits, Geometric learning for Image Analysis, VI.C 202-031).

We thank the students of the modeling week project ``Geodesic Tracking of Blood Vessels in Wide-Field Retinal Images'' for preliminary experimental results.
%


%
%
%
%
\bibliographystyle{splncs04}
\bibliography{references}

\end{document}

%% file: Sections/0_Abstract.tex
\begin{abstract}
%
In image analysis one often encounters spherical images, for instance in retinal imaging. 
The behavior of the vessels in the retina is an indicator of several diseases.
To automate disease diagnosis using retinal images,  it is necessary to develop an algorithm that automatically identifies and tracks vessels. 
To deal with crossings due to projected blood vessels in the image it is common to lift retinal images to the space of \emph{planar} positions and orientations $\M \coloneqq \mathbb{R}^2 \times S^1$. This implicitly assumes that the flat image accurately represents the geometry of the retina.
As the eyeball is a sphere (and not a plane), we propose to compute the \emph{cusp-free, crossing-preserving} geodesics in the space of \emph{spherical} positions and orientations $\tildeM$ on wide-field images. We clarify how to relate both manifolds and compare the calculated geodesics. The results show clear advantages of crossing-preserving tracking in $\tildeM$ over non-crossing-preserving tracking in $\tildeM$ and are comparable to tracking results in $\M$. 


\keywords{Geodesic Tracking  \and Vascular Tree Tracking \and Spherical Images \and Retinal Imaging \and Sub-Riemannian Geometry.}
\end{abstract}

%% file: Sections/1_Introduction.tex
\section{Introduction}
The eye offers a window to a person's health. Since the vessels on the retina behave similarly to the vessels throughout the rest of the body, taking pictures of the retina allows for a noninvasive way to diagnose and monitor several diseases, including diabetes, hypertension, and Alzheimer's disease \cite{colligris2018ocular,sasongko2016retinal,weiler2015arteriole}. 
Automatic vessel tracking algorithms help the efficient diagnosis of these diseases; we consider geodesic tracking methods which calculate the shortest path, following the biological structure, between two points on the same vessel.

There has been a lot of research on automatic tracking of blood vessels on retinal images \cite{Liu2021tubular,benmansour,bekkersPhD}. These algorithms often lift the image from $\Rtwo$ to 
the space of planar positions and orientations $\M \coloneqq \Rtwo \times S^1$. The lifting step disentangles difficult crossing structures, resulting in better tracking results. Recent works show that 
data-driven sub-Riemannian (SR)
metric tensor fields on $\M$ result in good tracking results for entire vascular trees \cite{berg2024geodesic}.

These existing models implicitly assume that the retinal image accurately captures the geometry of the retina. However, since the image is flat while the retina is spherical, the distances in the projection deviate from the actual distances. 
This deviation influences the tracking results. 
This concern was first put forward 
by Mashtakov et al. \cite{mashtakov2017tracking}: 
processing spherical images requires 
data-driven versions of 
SR \cite{boscainSO3,beshastniy,Beretovskii2016SubRiemannianDistance} geodesics 
on the space of \emph{spherical} positions and orientations $\tildeM \equiv \SO(3)$ instead of on $\M$. 
The SR geometry can naturally be visualised (as pointed out by Boscain \& Rossi \cite{boscainSO3}) with
a Reeds-Shepp car traveling over the blood vessels on the retinal sphere, cf.~Fig.~\ref{fig:Intro}, where the car can only move forward and backward and turn the wheel with finite costs, while sideward motions have infinite cost. 

A limitation of SR geometry for vascular tracking is the presence of cusps, which do not occur in blood vessels \cite{boscainSO3}, \cite[Fig.~8]{mashtakov2017tracking}. This was solved in \cite{duits2018optimal} by introducing the sub-Finslerian (SF) forward gear Reeds-Shepp car model. Since the car in Fig.~\ref{fig:Intro} now also cannot move backward, it cannot create cusps.

\begin{figure}
    \centering
        \begin{tikzpicture}[scale=1]
            \node[] (pic) at (0,0) {\includegraphics[width=0.8\textwidth]{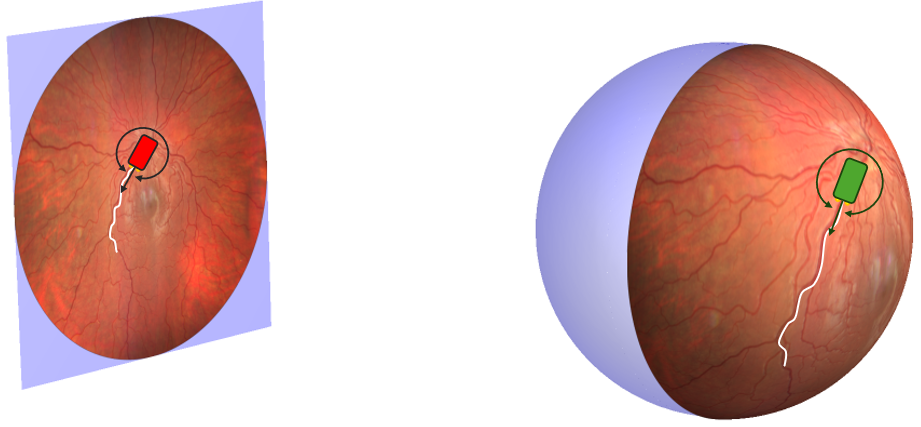}};
            \node[] at (-2.35,2) {$\M$};
            \node[] at (3,2) {$\tildeM$};
            \draw[stealth-] (-1.5,0.3) -- (0.5,0.3) node[midway,above] {$\Pi$};
        \end{tikzpicture}
    \vspace{-0.2cm}
\caption{Visualization of a tracking result of a blood vessel (white) on the plane (left) and on the sphere (right). We also visualize cars (red on the plane, green on the sphere) that can only move forwards and change orientation.
This gives intuition behind existing models on $\M \equiv \SE(2)$ \cite{reeds1990optimal,duits2018optimal} (red car), and our proposed, and previous \cite{mashtakov2017tracking} models, 
on $\tildeM \equiv \SO(3)$ (green car). We show that the projection map $\Pi$ is uniquely determined, cf.~Thm.~\ref{thm:pi_map}.}
\label{fig:Intro}
\end{figure}
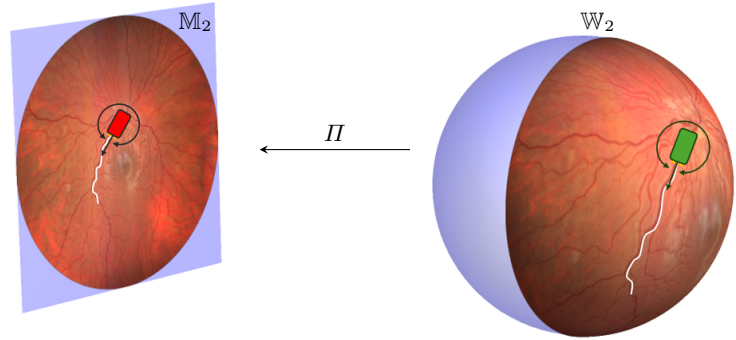

In geodesic vessel tracking algorithms \cite{caselles1997geodesic,Liu2021tubular,benmansour,bekkersPhD,mashtakov2017tracking,Chen2014VesselExtraction} one first computes a distance map as the viscosity solution of an eikonal PDE system, and 
then computes the \emph{globally optimal} geodesics by applying a 
steepest descent on the distance map. We will follow this approach and modify these existing geometric models, regarding the base manifold, model, and algorithm as we explain next in detail. \\[6pt]    
\textbf{Our Contributions. }
We propose a new data-driven extension of \cite{boscainSO3,beshastniy,Beretovskii2016SubRiemannianDistance} of the SR geodesics model on $\tildeM$, and
practically improve the model and algorithm of Mashtakov et al. \cite{mashtakov2017tracking}. More specifically, we do this by:
\begin{enumerate}
\item 
creating a \emph{crossing-preserving} tracking algorithm on $S^2$ using a 
cost function defined on the space of spherical positions and orientations $\tildeM$. \\[4pt]
We show examples where our new geodesic tracking method no longer takes wrong exits at crossings due to our crossing-preserving cost function, which is new compared to \cite{mashtakov2017tracking}.
\\[-6pt]
\item 
including a forward gear constraint, which is SF as opposed to SR, to get \emph{cusp-free} spherically projected $\tildeM$-geodesics. \\[4pt]
Such cusps \cite{boscainSO3}, \cite[Fig.~8]{mashtakov2017tracking} are cumbersome in our specific application. We provide a new spherical version of the forward gear Reeds-Shepp car \cite{duits2018optimal}. 
\\[-6pt]
\item 
applying the method to wide-field images \cite{Maastrichtstudy} instead of standard images. \\[4pt] 
They cover $120^{\circ}$ of the eyeball with 17~\!\% deformations, so that the difference between the $\tildeM$-model and its $\M$ approximation is much larger than on regular optical images (covering $72^\circ$ of the eye with 5~\!\% deformations \cite{mashtakov2017tracking}).
\\[-6pt]
\item computing the SR/SF distance maps with fast, simple, accurate GPU-code. \\[4pt]
We transfer the PDE-approach in \cite{bekkers2015datadriven} from the SR car model on $\SE(2) \equiv \M$ to the SF forward gear car model on $\tildeM$. 
Where advanced anisotropic fast-marching \cite{mirebeau2019fastmarching,mirebeau2022parallelfm} can approximate SR and SF constraints \cite{duits2018optimal}, this PDE-approach allows for \emph{exact} constraints.
The Python code for the experiments is released with this article. 
The PDE solver is accelerated by GPU parallelisation using the Python extension Taichi \cite{Taichi}.
\end{enumerate}
With these adaptations we can track vascular trees on wide-field images, with all SR or SF geodesics computed by steepest descent on a single distance map.

%% file: Sections/2_TrackingInSE2.tex
\section{Space of Planar Positions \& Orientations\texorpdfstring{ $\M$}{}}
In many existing methods \cite{bekkers2015datadriven,berg2024geodesic}, tracking on vascular images is done by lifting the image from $\R^2$ to the space of planar positions and orientations $\M$. It has the benefit of appropriately dealing with crossing and overlapping structures. 
Vessel tracking algorithms acting only on position space tend to take the wrong exit at crossing structures and one overcomes this \cite{berg2024geodesic,duits2018optimal,bekkers2015datadriven} by lifting the image data and the tracking to $\M$.  
As can be seen in Fig.~\ref{fig:orientation_score_draft},
crossing lines become disentangled in the lifted image data which enables correct geodesic tracking of the blood vessels.    
The 3D space $\M$ is defined as follows.
\begin{definition}[Space of planar positions and orientations]
    The space of planar positions and orientations $\M$ is defined as the smooth manifold 
    \begin{equation}\label{eq:def:M2}
        \M \coloneqq (T \Rtwo) \setminus \{0\} / \sim,
    \end{equation}
    where T and 0 denote the tangent bundle and the 0-section, respectively, and the equivalence relation $\sim$ is given, for $(\mathbf{x}_1, \dot{\mathbf{x}}_1), (\mathbf{x}_2, \dot{\mathbf{x}}_2) \in T \Rtwo \setminus \{0\}$, by
    \begin{equation*}
        (\mathbf{x}_1, \dot{\mathbf{x}}_1) \sim (\mathbf{x}_2, \dot{\mathbf{x}}_2) \iff \mathbf{x}_{1} = \mathbf{x}_2, \textrm{ and } \exists \lambda > 0: \dot{\mathbf{x}}_1 = \lambda \dot{\mathbf{x}}_2.
    \end{equation*}
    A roto-translation $(\mathbf{b}, R)$ in the special Euclidean group $\SE(2) \coloneqq \Rtwo \rtimes \SO(2)$ acts on $\mathbf{p} = (\mathbf{x}, \dot{\mathbf{x}}) \in \M$ by 
    group action
    \begin{equation}\label{eq:left_action_m2}
        L_{(\mathbf{b}, R)} (\mathbf{x}, \dot{\mathbf{x}}) \coloneqq (\mathbf{b} + R \mathbf{x}, R_* \dot{\mathbf{x}}), \textrm{ with push-forward } R_*.
    \end{equation}
    Upon choosing a reference point $\mathbf{p_0} \coloneqq ((0, 0), (1, 0)) \in \M$ we find
    \begin{equation}\label{eq:quotients}
        \M \equiv \SE(2) / \Stab_{\SE(2)}(\mathbf{p_0}) \equiv \SE(2) / \{(\mathbf{0}, I)\} \equiv \SE(2),
    \end{equation}
    from which we conclude that $\M$ is the principal homogeneous space of $\SE(2)$.
\end{definition}
It can be shown that $\M \equiv \Rtwo \times S^1$. Since $S^1 \equiv \R/(2\pi\mathbb{Z}) \equiv \SO(2)$, we identify
\begin{equation}\label{eq:identificationS1SO2}
(\cos(\theta), \sin(\theta)) \longleftrightarrow \theta \longleftrightarrow R_\theta\in \SO(2),
\end{equation}
where $R_\theta$ is a counter-clockwise rotation by $\theta$.
The above Lie group quotient identifications mean that points in $\M$ can 
be parametrised as $(\mathbf{x}, (\cos(\theta), \sin(\theta)))$, $(\mathbf{x}, \theta)$, or $(\mathbf{x}, R_\theta)$ for $\mathbf{x} \in \Rtwo$, $\theta \in \R/(2\pi\mathbb{Z})$.

On a principal homogeneous space like $\M$, it makes sense to perform processing in a \emph{left-invariant frame}. We choose the frame
\begin{equation}\label{eq:left_invariant_frame_m2}
\A_1|_\mathbf{p} \coloneqq (L_{g_\mathbf{p}})_* \partial_x|_\mathbf{p_0}, \; \A_2|_\mathbf{p} \coloneqq (L_{g_\mathbf{p}})_* \partial_y|_\mathbf{p_0}, \textrm{ and } \A_3|_\mathbf{p} \coloneqq (L_{g_\mathbf{p}})_* \partial_\theta|_\mathbf{p_0},
\end{equation}
where $(L_{g_\mathbf{p}})_*$ is the push-forward of left multiplication \eqref{eq:left_action_m2} with $L_{g_\mathbf{p}}$, and $g_\mathbf{p} = (\mathbf{x}, R_\theta)$ for $\mathbf{p} = (\mathbf{x}, \theta)$. Direct computation yields the (left-invariant) vector fields
$\A_1 = \cos(\theta) \partial_x + \sin(\theta) \partial_y$, $\A_{2} = -\sin (\theta) \partial_x + \cos (\theta) \partial_y$, and $\A_3 = \partial_\theta$.

We can lift an image from $\Rtwo$ to $\M$ with the orientation score transform:
\begin{definition}[Lifting the Data: Orientation Score]\label{def:orientation_score}
    The orientation score transform $W_\psi: \mathbb{L}_2(\Rtwo) \to \mathbb{L}_2(\M)$, using anisotropic wavelet $\psi$, maps an image $f \in \mathbb{L}_2(\Rtwo)$ to an orientation score $U=W_\psi f \in \mathbb{L}_2(\M)$ and is given by
    \begin{align*}
        (W_\psi f)(\mathbf{x},\theta) \coloneqq \int_{\Rtwo} \overline{\psi(R_\theta^{-1}(\mathbf{y} - \mathbf{x}))} f(\mathbf{y}) \mathrm{d}\mathbf{y}.
    \end{align*}
\end{definition}
We lift with cake wavelets $\psi$ \cite{franken2008horizontality} (cf. Figure \ref{fig:cakewavelet_draft}) as they do not tamper with the data evidence and allow for fast reconstruction of the original image $f$ by integration over $\theta$. The lifting disentangles crossing structures, see Figure~\ref{fig:orientation_score_draft}.
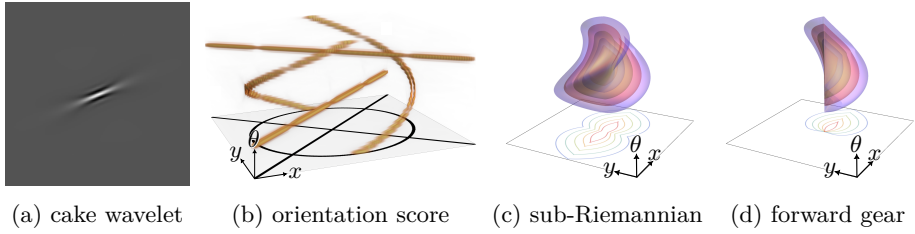
\begin{figure}
\centering
\begin{subfigure}[t]{0.2\textwidth}
\centering
\includegraphics[width=\textwidth]{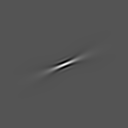}
\caption{cake wavelet}\label{fig:cakewavelet_draft}
\end{subfigure}
\hfill
\begin{subfigure}[t]{0.31\textwidth}
\begin{tikzpicture}
\draw(0, 0)node[inner sep=0]{\includegraphics[width=\textwidth]{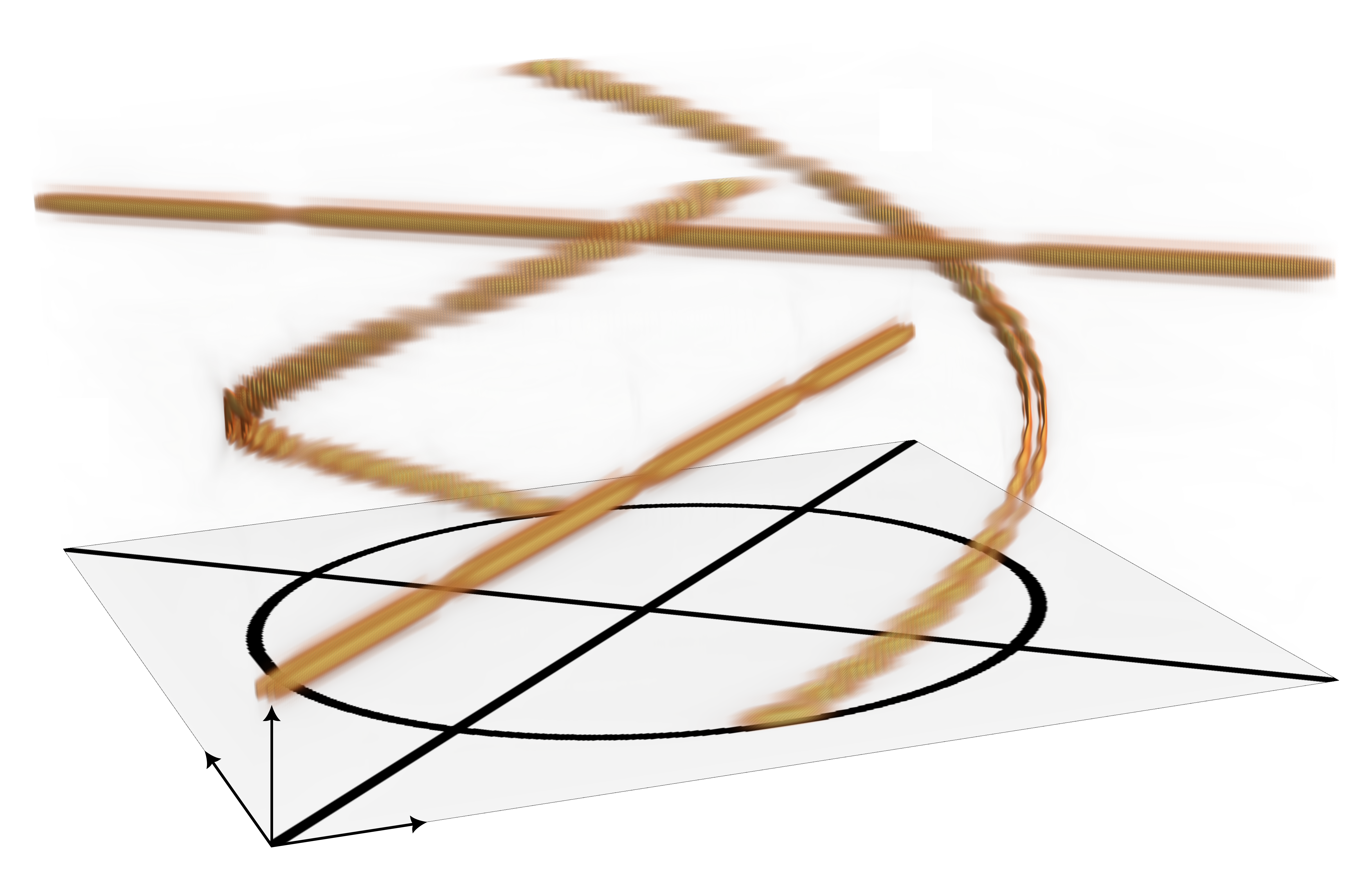}};
\draw(-0.6, -1.05)node{\contour{white}{$x$}};
\draw(-1.39, -0.8)node{\contour{white}{$y$}};
\draw(-1.15, -0.6)node{\contour{white}{$\theta$}};
\end{tikzpicture}
\caption{orientation score}\label{fig:orientation_score_draft}
\end{subfigure}
\hfill
\begin{subfigure}[t]{0.23\textwidth}
\begin{tikzpicture}
\draw(0, 0)node[inner sep=0]{\includegraphics[width=\textwidth]{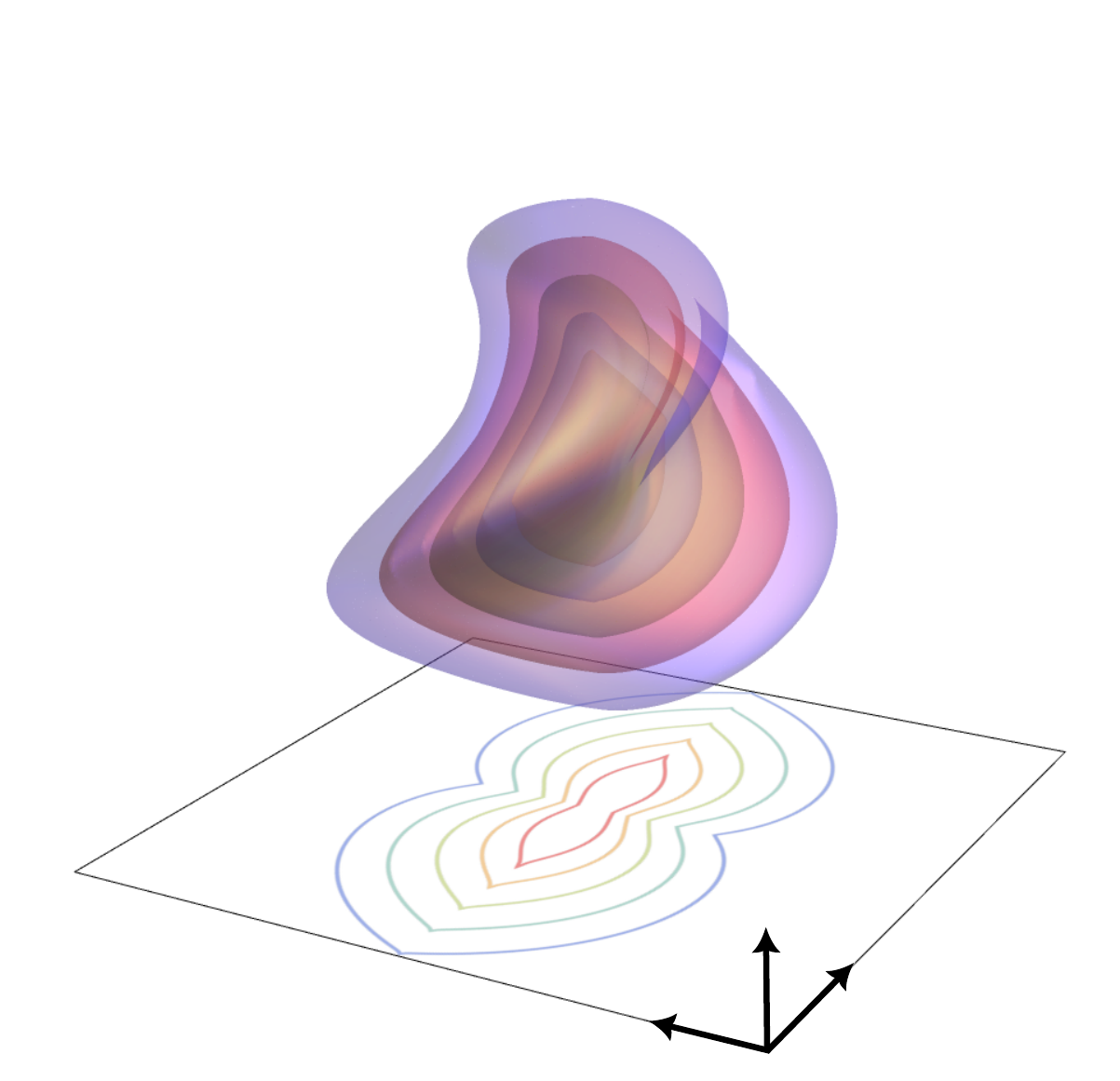}};
\draw(0.76, -0.99)node{\contour{white}{$x$}};
\draw(0.16, -1.19)node{\contour{white}{$y$}};
\draw(0.51, -0.83)node{\contour{white}{$\theta$}};
\end{tikzpicture}
\caption{sub-Riemannian}\label{fig:subriemannian_ball_se2_draft}
\end{subfigure}
\hfill
\begin{subfigure}[t]{0.23\textwidth}
\begin{tikzpicture}
\draw(0, 0)node[inner sep=0]{\includegraphics[width=\textwidth]{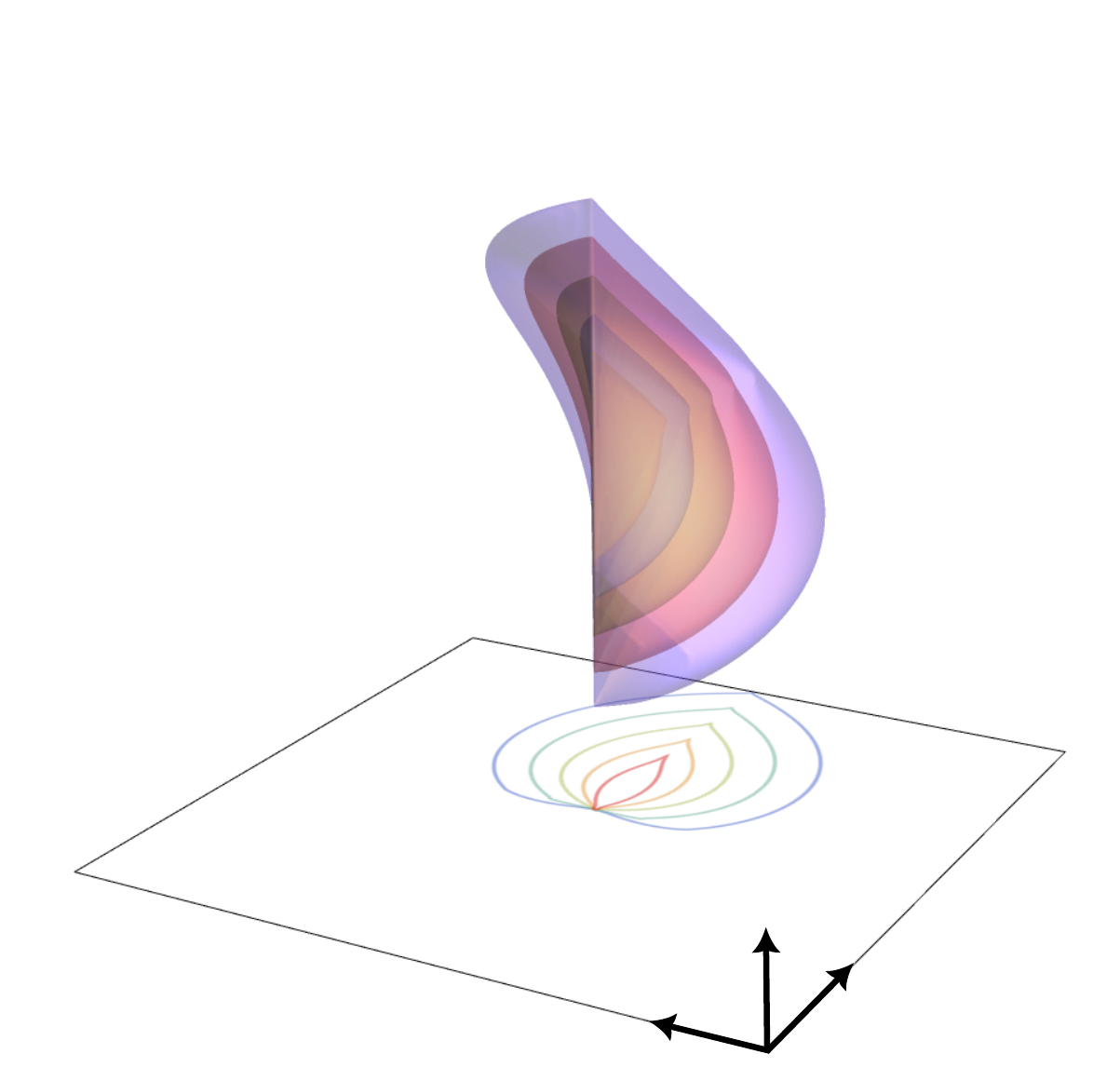}};
\draw(0.76, -0.99)node{\contour{white}{$x$}};
\draw(0.16, -1.19)node{\contour{white}{$y$}};
\draw(0.51, -0.83)node{\contour{white}{$\theta$}};
\end{tikzpicture}
\caption{forward gear}\label{fig:finslerian_ball_se2_draft}
\end{subfigure}
\caption{Left to right: Cakewavelet with orientation $\theta = \pi / 8$; Orientation score transform. The crossing structures in the image get disentangled by the lifting; Isosurfaces of sub-Riemannian and forward gear distances \eqref{eq:distance_m2} with $C \equiv 1$. Contours show the min-projection over orientations $\theta$.}
\label{fig:M2_balls_draft}
\end{figure}
Vector field $\A_1$ now points along the local orientation in orientation scores constructed using cake wavelets. This brings us to the concept of horizontality.
\begin{definition}[Horizontality on $\M$]\label{def:horizontality_m2}
On $\M$, we choose the \emph{distribution} $\Delta^{\M} \coloneqq \Span\{\A_1, \A_3\} \subset T \M$. Then a smooth curve $\gamma: \R \to \M$ is said to be \emph{horizontal} if $\dot{\gamma}(t) \in \Delta_{\gamma(t)}^{\M}$ for all $t$.
\end{definition}
By construction, the orientation score transform with cake wavelets lifts curves on $\Rtwo$ to horizontal curves on $\M$ \cite[Sec.~2.8.5]{franken2008horizontality}. If we define $\{\omega^i\}_{i = 1}^3$ to be the dual frame of $\{\A_i\}_{i = 1}^3$, i.e. $
\omega^i(\mathcal{A}_{j}) = \delta^{i}_{j}$, then the tensor fields $\omega^1 \otimes \omega^1$, $\omega^2 \otimes \omega^2$, and $\omega^3 \otimes \omega^3$ measure forward, lateral, and angular movement, respectively, letting us construct the following Reeds-Shepp car models \cite{reeds1990optimal,duits2018optimal}:
\begin{equation}\label{eq:subriemannian_m2}
\mathcal{G}_\mathbf{p} = 
\begin{cases}
    C^2(\mathbf{p}) \left(\xi_{\M}^2 \omega^1 \otimes \omega^1 + \omega^3 \otimes \omega^3\right)|_\mathbf{p}, & \textrm{on } \Delta^{\M}_\mathbf{p} \times \Delta^{\M}_\mathbf{p}, \\
    +\infty, & \textrm{else, and}
\end{cases}
\end{equation}
\begin{equation}\label{eq:finsler_m2}
\left|\mathcal{F}(\mathbf{p}, \cdot)\right|^2 =
\begin{cases}
    C^2(\mathbf{p}) \left(\xi_{\M}^2 \omega^1 \otimes \omega^1 + \omega^3 \otimes \omega^3\right)|_\mathbf{p}, & \textrm{on } \Delta^{\M,+}_\mathbf{p}, \\
    +\infty, & \textrm{else},
\end{cases}
\end{equation}
where $C: \M \to \R_{> 0}$ is a positive cost function, $\xi_{\M}$ is the stiffness parameter regulating the relative cost of moving forwards compared to turning, and $\Delta^{\M,+}_\mathbf{p} \coloneqq \{\dot{\mathbf{p}} \in \Delta^{\M}_\mathbf{p} \mid \omega^1(\dot{\mathbf{p}}) \geq 0\}$. 
\begin{remark}\label{rem:explanation_cars}
One can understand the car models by considering Fig.~\ref{fig:Intro}. The car,
which only has a forward gear,
can just move forwards and turn the wheel. This corresponds to $\mathcal{F}$ in \eqref{eq:finsler_m2},
since curves with tangents in $\Delta^{\M,+}_\mathbf{p}$ only point forwards and angularly.
As such, we call $\mathcal{F}$ the \emph{forward gear} model. For the model $\mathcal{G}$ in \eqref{eq:subriemannian_m2}, the car can also move backwards. We call $\mathcal{G}$ the \emph{sub-Riemannian} model.
In both models, the cost function ensures the car remains on the vasculature.
\end{remark}
We can use these models to define the \emph{geodesic distances} (shown in Figures~\ref{fig:M2_balls_draft}c-d): 
\begin{equation}\label{eq:distance_m2} 
d_\mathcal{G}(\mathbf{p}, \mathbf{q}) = \inf_{\gamma \in \Gamma} \int_0^1\!\! \sqrt{\mathcal{G}_{\gamma(t)}(\dot{\gamma}(t), \dot{\gamma}(t))} \diff t, \ 
d_\mathcal{F}(\mathbf{p}, \mathbf{q}) = \inf_{\gamma \in \Gamma} \int_0^1 \!\! \mathcal{F}(\gamma(t), \dot{\gamma}(t)) \diff t,
\end{equation}
where $\Gamma \coloneqq \{\gamma: [0,1] \to \M \textrm{ piecewise } C^1 \mid \gamma(0) = \mathbf{p}, \gamma(1) = \mathbf{q}\}$;
the minimising curves are called \emph{geodesics}. For vascular tracking, we use the crossing-preserving vesselness $\mathcal{V}: \M \to \R_{> 0}$ from \cite{berg2024geodesic}. The cost function, with $\lambda, p > 0$, is given by
\begin{equation}\label{eq:cost_m2}
    C^{\M}(\mathbf{p}) \coloneqq \frac{1}{1 + \lambda\,\vert\mathcal{V}(\mathbf{p})\vert^p}.
\end{equation}

%% file: Sections/3_TrackingInSO3.tex
\section{Space of Spherical Positions \& Orientations\texorpdfstring{ $\tildeM$}{}}
The retina itself is spherical; in this work we therefore track vascular trees in the space of \emph{spherical} positions and orientations $\tildeM$, which is defined as follows
\begin{definition}[Space of spherical positions and orientations]
    The space of spherical positions and orientations $\tildeM$ is defined as the smooth manifold
    \begin{align}
        \tildeM \coloneqq (T S^2) \setminus \{0\} /\sim,
    \end{align}
    where T and 0 denote the tangent bundle and the 0-section respectively and the equivalence relation $\sim$ is given, for $(\mathbf{n}_1, \dot{\mathbf{n}}_1), (\mathbf{n}_2, \dot{\mathbf{n}}_2) \in T S^2 \setminus \{0\}$, by
    \begin{equation*}
        (\mathbf{n}_1, \dot{\mathbf{n}}_1) \sim (\mathbf{n}_2, \dot{\mathbf{n}}_2) \iff \mathbf{n}_{1} = \mathbf{n}_2, \textrm{ and } \exists \lambda > 0: \dot{\mathbf{n}}_1 = \lambda \dot{\mathbf{n}}_2.
    \end{equation*}
    A rotation $R$ in the three-dimensional special orthogonal group $\SO(3)$ acts on $(\mathbf{n}, \dot{\mathbf{n}}) \in \tildeM$ via the group action
    \begin{equation}\label{eq:left_action_w2}
        R . (\mathbf{n}, \dot{\mathbf{n}}) \coloneqq (R \mathbf{n}, R_* \dot{\mathbf{n}}), \textrm{ with push-forward } R_*.
    \end{equation}    
    Choosing a reference point $\mathbf{q_0} = ((1,0,0), (0,0,1)) = (\mathbf{n}_0, \dot{\mathbf{n}}_0) \in \tildeM$
    we find
    \begin{equation}\label{eq:quotients_w2}
    \tildeM \equiv \SO(3) / \Stab_{\SO(3)}(\mathbf{q_0}) \equiv \SO(3) / \{I\} \equiv \SO(3),
    \end{equation}
    from which we  conclude that $\tildeM$ is the principal homogeneous space of $\SO(3)$.
\end{definition}
\begin{figure}
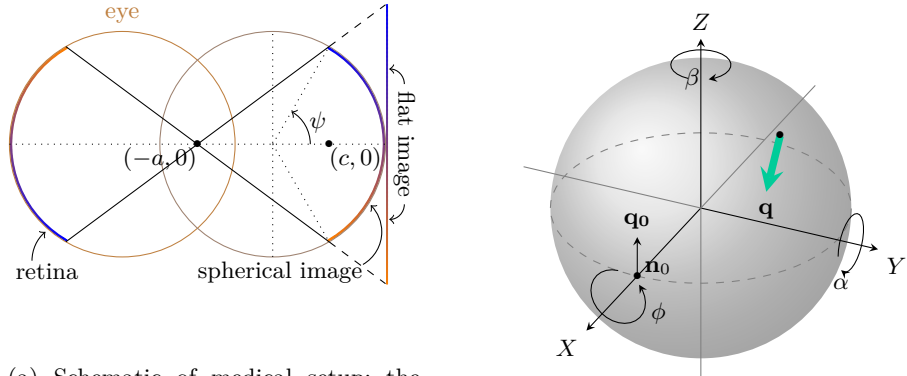

    \centering
    \begin{subfigure}[b]{0.45\textwidth}
        \centering
        \include{Figures/FlatvsCurved_Intro}
        \caption{Schematic of medical setup: the curved retina is captured in a flat image, which we map to a curved image with the same geometry as the retina.}
        \label{fig:Intro:Setting}
    \end{subfigure}
    \hfill
    \begin{subfigure}[b]{0.45\textwidth}
        \centering
        \include{Figures/SO3Coordinates}
        \caption{Coordinates in $\SO(3) \equiv \tildeM$. 
        }
    \label{fig:so3_coordinates}
    \end{subfigure}
    \caption{Visualization of the eye and the captured spherical image, and the used coordinates $\{\alpha, \beta, \phi\}$ of the green Reeds-Shepp car $\mathbf{q} =
R(\alpha, \beta, \phi). \mathbf{q}_0$ with $\mathbf{q}_0=(\mathbf{n}_0,\dot{\mathbf{n}}_0) \in T(S^2)$ on the retina (recall Fig.~\ref{fig:Intro:SettingAndCoordinates}). Parameter $\psi$ denotes the max. angle of the wide-field image, $c>0$ is the distance from the eyeball centre to the flat image, $(-a,0)$ is the focal point. }
    \label{fig:Intro:SettingAndCoordinates}
\end{figure}
We use the following coordinates on $\SO(3)$ (see Figure \ref{fig:so3_coordinates}):
\begin{equation*}
R(\alpha, \beta, \phi) \longleftrightarrow (\alpha, \beta, \phi) \iff R = R_Z(-\beta) \after R_Y(-\alpha) \after R_X(\phi),
\end{equation*}
with $R_X(t), R_Y(t), R_Z(t)$ 
counter-clockwise rotations around the $X$-, $Y$-, and $Z$-axes by $t$ radians. 
Since $\tildeM \equiv \SO(3)$, we use the same coordinates thereon.

As with $\M$, we can get left-invariant vector fields on $\tildeM$ by pushing forward derivatives at the reference point $\mathbf{q_0}$ with the group action; we choose the frame
\begin{equation}\label{eq:left_invariant_frame_w2}
\B_1|_\mathbf{q} \coloneqq (R_\mathbf{q})_*\partial_\alpha|_\mathbf{q_0}, \B_2|_\mathbf{q} \coloneqq (R_\mathbf{q})_*\partial_\beta|_\mathbf{q_0}, \textrm{ and } \B_3|_\mathbf{q} \coloneqq (R_\mathbf{q})_*\partial_\phi|_\mathbf{q_0},
\end{equation}
where $R_*$ is the push-forward of left multiplication \eqref{eq:left_action_w2} with $R$, and $R_\mathbf{q} = R(\alpha, \beta, \phi)$ for $\mathbf{q} \in \tildeM$ with coordinates $(\alpha, \beta, \phi)$.

We want to perform crossing-preserving tracking on $\tildeM$ instead of $\M$. For this, we need to define $\tildeM$ analogues of the $\M$ curve optimizations in \eqref{eq:distance_m2}, which in turn requires a cost function $C$ on $\tildeM$. Since the tracking must be crossing-preserving, this cost must be computed from lifted data. We achieve this by mapping the crossing-preserving vesselness from $\M$ to $\tildeM$. Mashtakov et al. \cite[Sec.~1.1]{mashtakov2017tracking} already derived a map between $S^2$ and $\Rtwo$ corresponding to the experimental setup (Figure~\ref{fig:Intro:Setting}); we extend this map $\Pi$ (cf.~Figure~\ref{fig:Intro}) from $\tildeM$ to $\M$. In Theorem~\ref{thm:pi_map}, we prove that this extension is uniquely defined by requiring that horizontal curves in $\tildeM$ get mapped to horizontal curves in $\M$.
\begin{definition}[Horizontality on $\tildeM$]\label{def:horizontality_w2}
On $\tildeM$, we choose the \emph{distribution} $\Delta^{\tildeM} \coloneqq \Span\{\B_1, \B_3\} \subset T \tildeM$. Then a smooth curve $\delta: \R \to \tildeM$ is said to be \emph{horizontal} if $\dot{\delta}(t) \in \Delta_{\delta(t)}^{\tildeM}$ for all $t$.
\end{definition}
The choice of distribution $\Delta^{\tildeM} \coloneqq \Span\{\B_1, \B_3\} \subset T \tildeM$ is such that the intuitive relationship between the left-invariant frame and the orientation score carries over from $\M$ to $\tildeM$ (cf. $\Delta^{\M} \coloneqq \Span\{\A_1, \A_3\}$). 

In our coordinates, the transition map $\pi:D(\pi)\to \mathbb{R}^2$ of \cite{mashtakov2017tracking} is given by 
\begin{equation}\label{eq:pi_forward}
\pi
(\alpha, \beta) = \left(\frac{(a + c) \sin(\alpha)}{a + \cos(\alpha) \cos(\beta)}, \frac{(a + c) \cos(\alpha) \sin(\beta)}{a + \cos(\alpha) \cos(\beta)}\right),
\end{equation}
with domain $D(\pi) = \{R(\alpha, \beta, 0) \, \mathbf{n}_0 \in S^2 \mid \alpha, \beta \in (-\pi/2, \pi/2)\} \subset S^2$.

Here $a, c > 0$ are parameters of the experimental setup, see Figure \ref{fig:Intro:Setting}.
We can then uniquely extend $\pi$ to $\tildeM$ by requiring the preservation of horizontality:
\begin{theorem}\label{thm:pi_map}
Let $\Pi: D(\Pi) \to \M$, with $D(\Pi) \coloneqq D(\pi) \times S^1 $, such that a) $\pi$ is the spatial projection of $\Pi$,
and b) for any horizontal curve $\delta$ on $\tildeM$ the curve $\gamma \coloneqq \Pi \after \delta$ is horizontal on $\M$. Then, $\Pi$ is uniquely defined and is given by
\begin{equation}\label{eq:Pi_forward}
\Pi(\alpha, \beta, \phi) = (\pi(\alpha, \beta), \arg(\dot{\pi}^1(\alpha, \beta) + i \,\dot{\pi}^2(\alpha, \beta)))
\end{equation}
where, for $\pi^i$ the $i$-th component of $\pi$, we have
\begin{equation*}
\dot{\pi}^i(\alpha, \beta) 
\coloneqq \frac{\partial \pi^i}{\partial \alpha}(\alpha, \beta) \cos(\phi) + \frac{\partial \pi^i}{\partial \beta}(\alpha, \beta) \frac{\sin(\phi)}{\cos(\alpha)}.
\end{equation*}
\end{theorem}
\begin{proof}
We start by finding sufficient and necessary conditions for curves to be horizontal. Suppose $(x, y, \theta) \equiv \gamma: \R \to \M$ is horizontal. Then, we have $\dot{\gamma}(t) \in \Span\{\mathcal{A}_1, \mathcal{A}_3\}$, which is true if and only if
\begin{align}
& \dot{\gamma}(t) = c^1(t) \mathcal{A}_1 + c^3(t) \mathcal{A}_3 = \underbrace{c^1(t) \cos(\theta(t)) \partial_x}_{\dot{x}(t)} + \underbrace{c^1(t) \sin(\theta(t)) \partial_y}_{\dot{y}(t)} + \underbrace{c^3(t) \partial_\theta}_{\dot{\theta}(t)} \nonumber \\
& \iff \angle (\dot{x}(t), \dot{y}(t)) = \theta(t). \label{eq:angle_m2}
\end{align}
Similarly, suppose $(\alpha, \beta, \phi) \equiv \delta: \R \to \tildeM$ is horizontal. Then, we have $\dot{\delta}(t) \in \Span\{\mathcal{B}_1, \mathcal{B}_3\}$, which is true if and only if
\begin{align}
& \dot{\delta}(t) = c^1(t) \mathcal{B}_1 + c^3(t) \mathcal{B}_3 = \underbrace{c^1(t) \cos(\phi(t)) \partial_\alpha}_{\dot{\alpha}(t)} + \underbrace{c^1(t) \frac{\sin(\phi(t))}{\cos(\alpha(t))} \partial_\beta}_{\dot{\beta}(t)} + \underbrace{c^3(t) \partial_\phi}_{\dot{\phi}(t)} \nonumber\\
& \iff \angle (\dot{\alpha}(t), \dot{\beta}(t) \cos(\alpha(t))) = \phi(t).\label{eq:angle_w2}
\end{align}
Now, let $(\alpha_0, \beta_0, \phi_0) \in D(\Pi)$, and let $(\alpha, \beta, \phi) \equiv \delta: \R \to \tildeM$ be a horizontal curve with $\delta(0) = (\alpha_0, \beta_0, \phi_0)$. Define $\gamma \coloneqq \Pi \after \delta \equiv (x, y, \theta)$. Then, by a) we have that $(x(t), y(t)) = \pi(\alpha(t), \beta(t))$, so in particular $(x_0, y_0) = \pi(\alpha_0, \beta_0)$. We next impose the horizontality constraint b):
\begin{equation*}
\begin{split}
\theta(t) & \overset{\eqref{eq:angle_m2}}{=} \angle (\dot{x}(t), \dot{y}(t)) = \arg(\dot{x}(t) + i \dot{y}(t)) \\
& \overset{\eqref{eq:angle_w2}}{=} \arg\left(\left(\frac{\partial \pi^1}{\partial\alpha}(\alpha(t), \beta(t))\cos(\phi(t)) + \frac{\partial \pi^1}{\partial\beta}(\alpha(t), \beta(t))\frac{\sin(\phi(t))}{\cos(\alpha(t))}\right) \right. \\
& + \left. i \left(\frac{\partial \pi^2}{\partial\alpha}(\alpha(t), \beta(t))\cos(\phi(t)) + \frac{\partial \pi^2}{\partial\beta}(\alpha(t), \beta(t))\frac{\sin(\phi(t))}{\cos(\alpha(t))}\right)\right).
\end{split}
\end{equation*}
We now simply evaluate at $t = 0$ to see $\theta_0 \coloneqq \theta(0) = \arg(\dot{\pi}^1(\alpha_0, \beta_0) + i \dot{\pi}^2(\alpha_0, \beta_0))$.
We have found that $\Pi(\alpha_0, \beta_0, \phi_0) = (x_0, y_0, \theta_0)$, which agrees with \eqref{eq:Pi_forward}.
\qed
\end{proof}
By choosing to map horizontal curves to horizontal curves, the intuition that $\A_1$
points parallel to the local orientation while $\A_2$ points laterally carries over onto $\tildeM$: $\B_1$ points parallel to the local orientation while $\B_2$ points laterally. Hence, we can find analogues of the $\M$ Reeds-Shepp car models \eqref{eq:subriemannian_m2} and \eqref{eq:finsler_m2} on $\tildeM$:
\begin{equation}\label{eq:subriemannian_w2}
\mathcal{G}_\mathbf{q} = 
\begin{cases}
    C^2(\mathbf{q}) \left(\xi_{\tildeM}^2 \nu^1 \otimes \nu^1 + \nu^3 \otimes \nu^3\right)|_\mathbf{q}, & \textrm{on } \Delta^{\tildeM}_\mathbf{q} \times \Delta^{\tildeM}_\mathbf{q}, \\
    +\infty, & \textrm{else, and}
\end{cases}
\end{equation}
\begin{equation}\label{eq:finsler_w2}
\left|\mathcal{F}(\mathbf{q}, \cdot)\right|^2 =
\begin{cases}
    C^2(\mathbf{q}) \left(\xi_{\tildeM}^2 \nu^1 \otimes \nu^1 + \nu^3 \otimes \nu^3\right)|_\mathbf{q}, & \textrm{on } \Delta^{\tildeM,+}_\mathbf{q}, \\
    +\infty, & \textrm{else},
\end{cases}
\end{equation}
where $\{\nu^i\}_{i = 1}^3$ is the dual frame to $\{\B_i\}_{i = 1}^3$, i.e. $\nu^i(\B_j) = \delta^i_j$, $C: \tildeM \to \R_{> 0}$ is a positive cost function, $\xi_{\tildeM}$ is the stiffness parameter regulating the relative cost of moving forwards compared to turning, and $\Delta^{\tildeM,+}_\mathbf{q} \coloneqq \{\dot{\mathbf{q}} \in \Delta^{\tildeM}_\mathbf{q} \mid \nu^1(\dot{\mathbf{q}}) \geq 0\}$. 
The explanation in Remark~\ref{rem:explanation_cars} carries over onto the sphere. As such, we again call $\mathcal{G}$ the \emph{sub-Riemannian} model and $\mathcal{F}$ the \emph{forward gear} model. These models induce geodesic distances analogous to \eqref{eq:distance_m2}.
For the cost function, we 
pull back the $\M$ cost function \eqref{eq:cost_m2} to $\W$: 
\begin{align}
    C^{\tildeM} \coloneqq C^{\M} \after \Pi.\label{eq:cost_w2}
\end{align}

%% file: Figures/FlatvsCurved_Intro.tex

\begin{tikzpicture}
    \draw [test={.5pt}{orange}{blue}, domain=120:240] plot ({1.48*cos(\x)}, {1.48*sin(\x)}) node at (-1,-1.7){retina}; 
    \draw[->] (-1.2,-1.55) arc[radius = 0.5, start angle = 180, end angle = 130];
    
    \draw[brown] (0,0) circle (1.5);
    \draw[brown] (0,1.5) node[above] {eye};

    \draw [test={.4pt}{orange}{blue}] (3.52,{-(1.5*sin(60)+(1.5*(1-cos(60))*(3*sin(60)/(2+3*cos(60))))}) -- (3.52,{1.5*sin(60)+(1.5*(1-cos(60))*(3*sin(60)/(2+3*cos(60)))}) node at (3.75,0)[opacity = 1, rotate = -90]{flat image};
    \draw [->] (3.75,-0.8) arc[radius = 0.2, start angle = 0, end angle = -90];
    \draw [->] (3.75,0.8) arc[radius = 0.2, start angle = 0, end angle = 90];
    
    \draw [test={.5pt}{orange}{blue}, domain=-60:60] plot ({2+1.48*cos(\x)}, {1.48*sin(\x)}) node at (2.1,-1.7)[opacity = 1, rotate = -0]{\small spherical image}; 
    \draw [->] (3.2,-1.6) arc[radius = 0.5, start angle = -60, end angle = 45];
    \draw[beaver] (2,0) circle (1.5);

    \draw[dotted] (-1.5,0) -- (3.5,0);
    \draw node at ({2+1.5*cos(60},0) {\textbullet} node at (3.1,-0.2) {$(c,0)$};
    \draw node at (1,0) {\textbullet} node at (0.5,-0.2) {$(-a,0)$};
    \draw[dotted] (2,-1.5) -- (2,1.5);
    \draw[dotted] (2,0) -- ({2+1.5*cos(60)},{0+1.5*sin(60)});
    \draw[dotted] (2,0) -- ({2+1.5*cos(60)},{0-1.5*sin(60)});
    \draw [->] (2.5,0.) arc[radius = 0.5, start angle = -0, end angle = 60] node at (2.6,0.3){$\psi$};

    \draw[domain=0:1] plot ({\x*1.5*cos(120)+(1-\x)*(2+1.5*cos(-60)}, {\x*(1.5*sin(120))+(1-\x)*1.5*sin(-60)});
    \draw[domain=0:1] plot ({\x*1.5*cos(240)+(1-\x)*(2+1.5*cos(60)}, {\x*(1.5*sin(240))+(1-\x)*1.5*sin(60)});
    \draw[dashed]  ({(2+1.5*cos(60)}, {1.5*sin(60)}) -- ({2+1.5},{1.5*sin(60)+(1.5*(1-cos(60))*(3*sin(60)/(2+3*cos(60)))});
    \draw[dashed]  ({(2+1.5*cos(-60)}, {1.5*sin(-60)}) -- ({2+1.5},{1.5*sin(-60)+(1.5*(1-cos(-60))*(3*sin(-60)/(2+3*cos(-60)))});
\end{tikzpicture}

%% file: Figures/SO3Coordinates.tex
\tikzmath{\x0 = 0; \y0 = 0; \x1 = 50; \y1 = -80; \t1 = -240; \r1 = 0.5; } 
\begin{tikzpicture}[tdplot_main_coords, scale = 2]

\coordinate (P) at ({1/sqrt(3)},{1/sqrt(3)},{1/sqrt(3)});

\shade[ball color = lightgray,
	opacity = 0.5
] (0,0,0) circle (1cm);

\tdplotsetrotatedcoords{0}{0}{0};
\draw[dashed,
	tdplot_rotated_coords,
	gray
] (0,0,0) circle (1);



\draw[-stealth,
        domain=0:320
] plot ({0.2*cos(\x)}, 1.1,{0.2*sin(\x)})
        node[below] {$\alpha$};

\draw[-stealth,
        domain=0:320
] plot ({0.2*cos(-\x)}, {0.2*sin(-\x)},1.1)
        node[left] {$\beta$};

\draw[-stealth,
        domain=90:400
] plot (1.3, {0.2*cos(\x)}, {0.2*sin(\x)})
        node at (1.3, {0.3*cos(350)}, {0.3*sin(350)}) {$\phi$};


\draw[-stealth] (0,0,0) -- (1.80,0,0) 
	node[below left] {$X$};

\draw[-stealth] (0,0,0) -- (0,1.30,0)
	node[below right] {$Y$};

\draw[-stealth] (0,0,0) -- (0,0,1.30)
	node[above] {$Z$};

\draw[gray] (0,0,0) -- (-1.80,0,0);

\draw[gray] (0,0,0) -- (0,-1.30,0);

\draw[gray] (0,0,0) -- (0,0,-1.30);

\draw[-stealth] (1,0,0) -- (1,0,0.3)
	node[above] {$\mathbf{q_0}$};
\draw node at ({cos(\x0)*cos(\y0)},{-cos(\x0)*sin(\y0)},{sin(\x0)}) {\textbullet} node at (1,0.15,0.1) {$\mathbf{n}_0$};

\draw[-stealth, line width=1mm, caribbeangreen] ({cos(\x1)*cos(\y1)},{-cos(\x1)*sin(\y1)},{sin(\x1)}) -- ({cos(\x1)*cos(\y1)+\r1*(-sin(\x1)*cos(\y1)*cos(\t1)-sin(\y1)*sin(\t1))},{-cos(\x1)*sin(\y1)+\r1*(sin(\x1)*sin(\y1)*cos(\t1)-cos(\y1)*sin(\t1))},{sin(\x1)+\r1*(cos(\x1)*cos(\t1))})
	node[below] {\textcolor{black}{$\mathbf{q}$}};
\draw node at ({cos(\x1)*cos(\y1)},{-cos(\x1)*sin(\y1)},{sin(\x1)}) {\textbullet};




\end{tikzpicture}


%% file: Sections/4_PDESolver.tex
\section{Eikonal PDE Solver}\label{sec:EikonalPDESolver}
We assume that vessels on the retina are well-modeled by geodesics (as e.g. \cite{bekkers2015datadriven,Liu2021tubular,mashtakov2017tracking}) of the distance maps induced by the car models \eqref{eq:subriemannian_m2}, \eqref{eq:finsler_m2}, \eqref{eq:subriemannian_w2}, and \eqref{eq:finsler_w2}; this allows us to find vessels by backtracking on these distance maps. We will now discuss how we compute the distance map and find corresponding geodesics.
Let $\mathcal{M} \in \{\M, \tildeM\}$, and let $\mathcal{F}$ be the (sub-)Finsler function corresponding to one of our controllers on $\mathcal{M}$.\footnote{The sub-Riemannian models also induce a Finsler function: $\mathcal{F}(\mathbf{p}, \dot{\mathbf{p}}) \coloneqq \sqrt{\mathcal{G}_\mathbf{p}(\dot{\mathbf{p}}, \dot{\mathbf{p}})}$.} Then, we are interested in computing the distance map with respect to some reference point $\mathbf{p}_0 \in \mathcal{M}$: $W(\mathbf{p}) \coloneqq d_\mathcal{F}(\mathbf{p}_0, \mathbf{p})$. The reference point $\mathbf{p}_0$ is called the \emph{seed} of the geodesic.
It was shown \cite{duits2018optimal} that $W$ is the (viscosity) solution of the eikonal PDE
\begin{equation}\label{eq:eikonal_pde}
\begin{cases}
\mathcal{F}^*(\mathbf{p}, \diff W(\mathbf{p})) = 1, \textrm{ on } \mathcal{M} \setminus \{\mathbf{p}_0\}, \\
W(\mathbf{p}_0) = 0,
\end{cases}
\end{equation}
where $\mathcal{F}^*$ is the dual Finsler function. 
Since \eqref{eq:eikonal_pde} is a boundary value problem, it is hard to solve. It can be solved efficiently with Fast Marching, e.g. Mirebeau \cite{mirebeau2019fastmarching}, though for technical reasons this requires sub-Riemannian and sub-Finslerian metrics to be approximated by highly anisotropic metrics. To be able to enforce exact sub-Riemannian and sub-Finslerian constraints, we follow Bekkers et al. \cite{bekkers2015datadriven} in approximating $W$ to accuracy $\epsilon > 0$ (in supremum norm) by iteratively solving a relaxed version of the eikonal PDE \eqref{eq:eikonal_pde}:
\begin{equation}\label{eq:iterative_pde}
\begin{cases}
\frac{\partial}{\partial r} W_{n + 1}^\epsilon(\mathbf{p}, r) = 1 - \mathcal{F}^*(\mathbf{p}, \diff W_{n + 1}^\epsilon(\mathbf{p}, r)), & \textrm{ on } \mathcal{M} \times [0, \epsilon], \\
W_{n + 1}^\epsilon(\mathbf{p}, 0) = W_n^\epsilon(\mathbf{p}, \epsilon), & \textrm{ on } \mathcal{M} \setminus \{\mathbf{p}_0\}, \\
W_{n + 1}^\epsilon(\mathbf{p}_0, r) = 0, & \textrm{ on } [0, \epsilon],
\end{cases}
\end{equation}
with as initial condition
\begin{equation}\label{eq:iterative_pde_initial_condition}
W_0^\epsilon(\mathbf{p}, 0) = \delta_{\mathbf{p}_0}^\mathcal{M}(\mathbf{p}) = \begin{cases}
    0, & \textrm{ if } \mathbf{p} = \mathbf{p}_0, \\
    + \infty, & \textrm{ else},
\end{cases}
\end{equation}
the morphological delta. 
In the left-invariant ($C \propto 1$), 
sub-Riemannian case on $\SE(2) \equiv \M$ it was shown for all $\mathbf{p} \in \SE(2)$ \cite[Thm.~E.3]{bekkers2015datadriven} that
\begin{equation}\label{eq:iterative_pde_convergence}
W(\mathbf{p}) = \lim_{\epsilon \downarrow 0} \lim_{n \to \infty} W_n^\epsilon(\mathbf{p}, 0).
\end{equation}
The proof of convergence, which relies on $\SE(2)$ morphological convolutions, can be readily extended to the left-invariant, sub-Riemannian case on $\tildeM \equiv \SO(3)$, by instead using $\SO(3)$ morphological convolutions. We would \emph{conjecture} that the distance map $W$ solves the eikonal PDE \eqref{eq:eikonal_pde} with any of the data-driven Finslerian car models considered in this work (and indeed on more general Finsler manifolds), and that the convergence result \eqref{eq:iterative_pde_convergence} extends to these situations; we hope to prove this in future work. 
Given (an approximation of) the distance map $W$, we compute the geodesic connecting point $\mathbf{p}$ to $\mathbf{p}_0$ by backtracking, i.e. by solving \cite[Sec.~2.5]{duits2018optimal}
\begin{equation}\label{eq:backtracking}
\begin{cases}
\dot{\gamma}(t) = -W(\mathbf{p})\, \diff_{\hat{\mathbf{p}}} \mathcal{F}^*(\gamma(t), \diff W(\gamma(t))), & \textrm{on } [0, 1], \\
\gamma(0) = \mathbf{p}, \gamma(1) = \mathbf{p}_0, &
\end{cases}
\end{equation}
where $\diff_{\hat{\mathbf{p}}} \mathcal{F}^*$ denotes the differential with respect to the second entry of $\mathcal{F}^*$. 

%% file: Sections/5_Experiments_NonAnonymous.tex
\section{Experimental Results}
We apply the new forward gear model on $\tildeM$ to track vessels on a wide-field image \cite{Maastrichtstudy} (max angle $2 \psi = 120^{\circ}$ cf. Fig.~\ref{fig:Intro:Setting}) of a retina (patch in Fig.~\ref{fig:UnderlyingImage}), using the cost function $C^{\tildeM}$ based on the crossing-preserving vesselness from \cite{berg2024geodesic}. We compare with two baselines: 1) tracking using the forward gear on $\M$ and the same crossing-preserving cost function, as in \cite{duits2018optimal}, to investigate the influence of the underlying manifold; 2) tracking using the forward gear on $\tildeM$ and an $\Rtwo$ cost function based on the Frangi vesselness \cite{Frangi2000Multiscale}, as in \cite{mashtakov2017tracking}, to investigate the influence of the crossing-preserving nature of the cost function. 
The implementations of the experiments (with parameters) can be found at \url{https://github.com/finnsherry/IterativeEikonal}.

We set one seed $\mathbf{p}_0$ near the optic disk. We use the eikonal PDE solver \eqref{eq:iterative_pde} discussed in Sec.~\ref{sec:EikonalPDESolver} to compute an estimate of the distance maps induced by our forward gear models (\eqref{eq:finsler_m2} and \eqref{eq:finsler_w2} on $\M$ and $\tildeM$, respectively). Afterwards, we calculate the geodesics by backtracking on the distance map, see \eqref{eq:backtracking}. 

In the experiments, we use 3 different cost functions: $C^{\M}$, $C^{\tildeM}$, and $C^{\tildeM,\Rtwo}$. The first cost function $C^{\M}$ is computed from the crossing-preserving vesselness $\mathcal{V}$ (see \cite{berg2024geodesic}) using \eqref{eq:cost_m2}. The second cost function $C^{\tildeM}$ is constructed from $C^{\M}$ using \eqref{eq:cost_w2} (we set $a = \frac{13}{21}$, $c = \frac{1}{2}$).
The final cost function $C^{\tildeM,\Rtwo}$ is based on the Frangi multiscale vessel enhancement filter $C^{\Rtwo}$ \cite{Frangi2000Multiscale}, as used in \cite{mashtakov2017tracking}: $C^{\tildeM,\Rtwo}(\alpha, \beta, \phi) \coloneqq C^{\Rtwo}(\pi(\alpha, \beta))$

Computing the distance maps and the geodesics on $\M$ and $\W$ is equally expensive, while pulling back the cost function from $\M$ to $\W$ requires negligible computational effort. 

\subsubsection{Baseline 1)}
We first compare the tracking results using the forward gear on $\M$ \eqref{eq:finsler_m2} with cost function $C^{\M} $ (Fig.~\ref{fig:TrackingSE2_CostSE2}) to the forward gear on $\W$ \eqref{eq:finsler_w2} with cost function $C^{\tildeM}$ (Fig.~\ref{fig:TrackingSO3_CostSO3}). 
The methods perform similarly, with neither taking ``shortcuts'' that are not present in the underlying vasculature, i.e. they do not jump between different vascular trees at crossings \cite{bekkers2015datadriven}. This is in line with \cite{mashtakov2017tracking}: since the cost function parameter $\lambda$ is large, it dominates the influence of the shape of the underlying manifold. In the white box, we see that the $\tildeM$ model follows the vasculature better than the $\M$ model.

\subsubsection{Baseline 2)}
We next compare the tracking results using the forward gear on $\W$ \eqref{eq:finsler_w2} with the cost function $C^{\tildeM,\Rtwo}$ (Fig.~\ref{fig:TrackingSO3_CostR2}) to the cost function $C^{\tildeM}$ (Fig.~\ref{fig:TrackingSO3_CostSO3}). We see that the model using $C^{\tildeM,\Rtwo}$ -- based on the Frangi filter which does not preserve crossings -- has some issues following the right vasculature. This is caused by the fact that the cost function does not distinguish between bifurcations and crossings: at bifurcations, vessels can quickly turn, whereas at crossings, turning should be prohibited.
Consequently, the model occasionally mistakes crossings for bifurcations, taking a ``shortcut'' (see bold dark green tract). On the other hand, the geodesics using $C^{\tildeM}$ better follow the vasculature, as can be seen in Fig.~\ref{fig:TrackingSO3_CostSO3}; the underlying crossing-preserving cost function has disentangled crossings, allowing the model to differentiate crossings from bifurcations and making it expensive to take a ``shortcut''.

\begin{figure}
    \centering
    \begin{subfigure}[t]{0.45\textwidth}
        \centering
        \begin{tikzpicture}[scale=1]
            \node[] (pic) at (0,0) {\includegraphics[width=\textwidth,trim={{0.3\textwidth} {1.1\textwidth} {0.7\textwidth} {0.5\textwidth}}, clip]{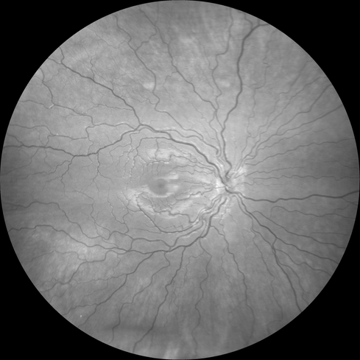}};
        \end{tikzpicture}
        \caption{Underlying image.}
        \label{fig:UnderlyingImage}
    \end{subfigure}
    \begin{subfigure}[t]{0.45\textwidth}
        \centering
        \begin{tikzpicture}[scale=1]
            \node[] (pic) at (0,0) {\includegraphics[width=\textwidth,trim={{0.3\textwidth} {1.1\textwidth} {0.7\textwidth} {0.5\textwidth}}, clip]{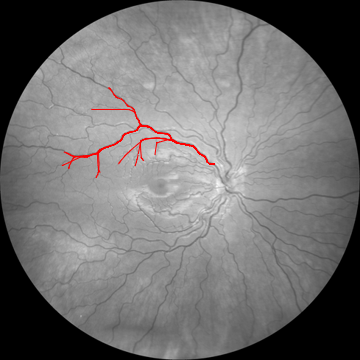}};
            \draw[white,thick] (-1,-1) rectangle (1,0);
        \end{tikzpicture}
        \caption{Tracking in $\M$ with cost function $C^{\M}$ and $\xi_{\M}=4$.}
        \label{fig:TrackingSE2_CostSE2}
    \end{subfigure}
    \begin{subfigure}[t]{0.45\textwidth}
        \centering
        \begin{tikzpicture}[scale=1]
            \node[] (pic) at (0,0) {\includegraphics[width=\textwidth,trim={{0.3\textwidth} {1.1\textwidth} {0.7\textwidth} {0.5\textwidth}}, clip]{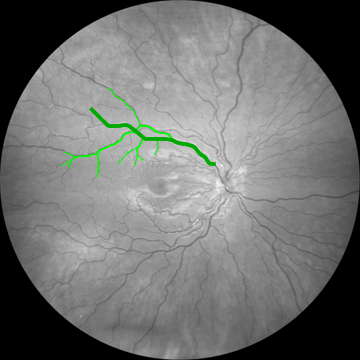}};
        \end{tikzpicture}
        \caption{Tracking in $\tildeM$ with cost function $C^{\tildeM,\Rtwo}$ and $\xi_{\tildeM}=6$.}
        \label{fig:TrackingSO3_CostR2}
    \end{subfigure}
    \begin{subfigure}[t]{0.45\textwidth}
        \centering
        \begin{tikzpicture}[scale=1]
            \node[] (pic) at (0,0) {\includegraphics[width=\textwidth,trim={{0.3\textwidth} {1.1\textwidth} {0.7\textwidth} {0.5\textwidth}}, clip]{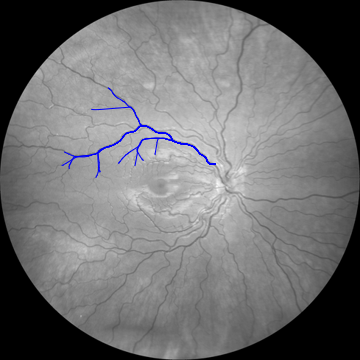}};
            \draw[white,thick] (-1,-1) rectangle (1,0);
        \end{tikzpicture}
        \caption{Tracking in $\tildeM$ with cost function $C^{\tildeM}$ and $\xi_{\tildeM}=6$.}
        \label{fig:TrackingSO3_CostSO3}
    \end{subfigure}
    \caption{Comparison tracking in $\M$ and $\tildeM$ with cost function in $\Rtwo$ vs $\M$. 
    The tracking results c), d) on $\tildeM$ perform better than 
    on $\M$ (white box) and using a lifted cost function d) avoids wrong exits c) at crossings (bold tract).}
    \label{fig:trackingInSE2vsSO3}
\end{figure}

%% file: Sections/6_Conclusion.tex
\subsubsection{Conclusion.}
We have introduced a model on the lifted space of spherical positions and orientations $\tildeM$ that accounts for local angular information. We have extended the model introduced in \cite{mashtakov2017tracking} to deal with wide-field retinal images and to induce cusp-free geodesics. Additionally, we identified the coordinate mapping between $\tildeM$ and the space of planar positions and orientations $\M$. This mapping is constructed in such a way that horizontal curves in $\tildeM$ map to horizontal curves in $\M$, using the same spatial coordinate mapping introduced in \cite{mashtakov2017tracking}. The coordinate mapping allows us to pull back the cost function from $C^{\M}$ to $C^{\tildeM}$.

We have validated the effectiveness of the extension of the coordinate mapping in the experimental section. We conclude that tracking models in the space of spherical positions and orientations $\tildeM$ perform better when using a crossing-preserving cost function, which differentiates between structures at crossings of blood vessels. Additionally, we found that the tracking results (on wide-field images) in $\tildeM$ can improve upon to those calculated in the space of planar positions and orientations $\M$.

